\documentclass[12pt,fleqn]{amsart}

\usepackage{amssymb}

\usepackage{color}
\newtheorem{theorem}{Theorem}[section]
\newtheorem{lemma}[theorem]{Lemma}

\newtheorem{problem}[theorem]{Problem}
\newtheorem{corollary}[theorem]{Corollary}
\theoremstyle{definition}

\theoremstyle{remark}
\newtheorem{remark}[theorem]{Remark}
\numberwithin{equation}{section}

\begin{document}

\title[Tightness of spaces of measures]{Countable tightness in the  spaces\\ of regular probability measures}

\author[G.\ Plebanek]{Grzegorz Plebanek}
\address{Instytut Matematyczny, Uniwersytet Wroc\l awski}
\email{grzes@math.uni.wroc.pl}

\author[D.\ Sobota]{Damian Sobota}
\address{Warsaw Center of Mathematics and Computer Science\newline \indent Institute of Mathematics and Computer Science, Wroc\l aw University of Technology}
\email{damian.sobota@pwr.edu.pl}
\thanks{G.\ Plebanek was partially supported by NCN grant (2014-2017).} 
\thanks{D.\ Sobota was partially supported by Warsaw Center of Mathematics and Computer Science.}




\begin{abstract}
We prove that if $K$ is a compact space and the space $P(K\times K)$ of regular probability measures on $K\times K$ has countable tightness in its $weak^*$ topology, then
$L_1(\mu)$ is separable for every $\mu\in P(K)$. It has been known that such a result is a consequence of Martin's axiom MA$(\omega_1)$.
 Our theorem has several consequences; in particular, it generalizes a theorem due to Bourgain and Todor\v{c}evi\'c
on measures on Rosenthal compacta.
\end{abstract}

\maketitle

\newcommand{\con}{\mathfrak c}
\newcommand{\eps}{\varepsilon}
\newcommand{\alg}{\mathfrak A}
\newcommand{\algb}{\mathfrak B}
\newcommand{\algc}{\mathfrak C}
\newcommand{\ma}{\mathfrak M}
\newcommand{\pa}{\mathfrak P}
\newcommand{\BB}{\protect{\mathcal B}}
\newcommand{\AAA}{\mathcal A}
\newcommand{\CC}{{\mathcal C}}
\newcommand{\FF}{{\mathcal F}}
\newcommand{\GG}{{\mathcal G}}
\newcommand{\LL}{{\mathcal L}}
\newcommand{\UU}{{\mathcal U}}
\newcommand{\VV}{{\mathcal V}}
\newcommand{\HH}{{\mathcal H}}
\newcommand{\DD}{{\mathcal D}}
\newcommand{\RR}{\protect{\mathcal R}}
\newcommand{\ide}{\mathcal N}
\newcommand{\btu}{\bigtriangleup}
\newcommand{\hra}{\hookrightarrow}
\newcommand{\ve}{\vee}
\newcommand{\we}{\cdot}
\newcommand{\de}{\protect{\rm{\; d}}}
\newcommand{\er}{\mathbb R}
\newcommand{\qu}{\mathbb Q}
\newcommand{\supp}{{\rm supp} }
\newcommand{\card}{{\rm card} }
\newcommand{\wn}{{\rm int} }
\newcommand{\ult}{{\rm ULT}}
\newcommand{\vf}{\varphi}
\newcommand{\osc}{{\rm osc}}
\newcommand{\ol}{\overline}
\newcommand{\me}{\protect{\bf v}}
\newcommand{\ex}{\protect{\bf x}}
\newcommand{\stevo}{Todor\v{c}evi\'c}
\newcommand{\cc}{\protect{\mathfrak C}}
\newcommand{\scc}{\protect{\mathfrak C^*}}
\newcommand{\lra}{\longrightarrow}
\newcommand{\sm}{\setminus}
\newcommand{\uhr}{\upharpoonright}

\newcommand{\sub}{\subseteq}
\newcommand{\ms}{$(M^*)$}
\newcommand{\m}{$(M)$}
\newcommand{\MA}{MA$(\omega_1)$}
\newcommand{\clop}{\protect{\rm Clop} }
\section{Introduction}

The \textit{tightness} of a topological space $X$, mentioned in the title and denoted here by $\tau(X)$, is the least cardinal number
such that for every $A\sub X$ and $x\in\ol{A}$ there is a set $A_0\sub A$ with $|A_0|\le \tau(X)$
and such that $x\in\ol{A_0}$.

In the sequel, $K$ always stands for a compact Hausdorff topological space.
 By $C(K)$ we denote the Banach
space of continuous functions on $K$ equipped with the supremum norm. As usual, the conjugate space
$C(K)^{*}$ is identified with $M(K)$, the space of signed Radon measures
on $K$ of finite variation. We denote by $P(K)$ the space of probability
Radon measures on $K$ and  consider $P(K)$ endowed with the $weak^*$ topology inherited from $C(K)^*$.

In the present paper we focus on the following problem.

\begin{problem}\label{mainproblem}
Suppose that $P(K)$ has countable tightness. Does this imply that every $\mu\in P(K)$ has countable Maharam type (that is,
$L_1(\mu)$ is separable)?
\end{problem}

There are several reasons why such a problem seems to be quite interesting and delicate. We now briefly outline some aspects of \ref{mainproblem}
and postpone a more detailed discussion to section \ref{section:5}.

Assuming Martin's axiom MA$(\omega_1)$,  Fremlin \cite{Fr97} showed that if a compact space $K$ admits a measure of uncountable type then $K$ can be continuously mapped onto
$[0,1]^{\omega_1}$, so in particular $K$ must have uncountable tightness. Since $P(K)$ contains a subspace homeomorhic to $K$ it follows that Problem \ref{mainproblem}
has a positive solution under  MA$(\omega_1)$.

Talagrand \cite{Ta81} showed that if $K$ admits a measure of type $\omega_2$ then $P(K)$ can be continuously mapped onto $[0,1]^{\omega_2}$. Thus the following analogue
of \ref{mainproblem} holds true:  if  $\tau(P(K)) \le \omega_1$ then every measure $\mu\in P(K)$ is of type $\le\omega_1$.

Pol \cite{Po82} investigated whether the following duality holds: $P(K)$ has countable tightness if and only if the Banach space $C(K)$ has property (C) of Corson.
By the main result of \cite{FPR} in order to verify such a duality it is sufficient to prove that Problem \ref{mainproblem} has a positive solution.

If $K$ is Rosenthal compact (i.e.\ $K$ can be represented as a pointwise compact space of Baire-one functions on some Polish space),
then every $\mu\in P(K)$ has countable type. This fact,  announced in Bourgain \cite{Bo76}, was proved by Todor\v{c}evi\'c \cite{To99} basing on
properties of Rosenthal compacta in forcing extensions; see also Marciszewski and Plebanek \cite{MP12}.  Problem \ref{mainproblem} would be a generalization of that result since
for Rosenthal compact $K$, the space $P(K)$ is also Rosenthal compact and consequently has countable tightness.

We shall prove below that for every compact space $K$, if $P(K\times K)$ has countable tightness then $K$ carries only measures of countable type. This does not solve
Problem \ref{mainproblem} completely, but seems to be a substantial step forward. After recalling basic definitions and facts in section \ref{section:2}, we prove in section \ref{section:3} some auxiliary results on measures on product spaces. In section \ref{section:4} we prove our main result --- Theorem \ref{mainthm}.
In the final section we present  some consequences of Theorem \ref{mainthm},  and state some  open problems related to our main topic.

\section{Preliminaries} \label{section:2}

For a given space $K$, $Bor(K)$ stands for the $\sigma$-algebra of all Borel subsets of $K$. Every $\mu\in P(K)$  is treated as a Borel measure on $K$, which is inner regular with respect to
compact sets. Recall that the $weak^*$ topology on $P(K)$ is the weakest one making a function $\mu\mapsto \int_K g\;{\rm d}\mu$ continuous for every $g\in C(K)$.

\begin{remark}\label{remark} Take an open set  $U\sub K$ and a closed set $F\sub K$. 
 Note that  the set of the form 
 $V_{U,a}=\{\nu\in P(K):  \nu(U)>a\}$
  is $weak^*$ open in $P(K)$ for every $a\in\er$. 
  
 Let  $M\sub P(K)$ and $\nu_0\in\overline{M}$. It follows that
 
 \begin{itemize}
 \item[(a)] $\nu_0(V)\le a$ provided $\nu(V)\le a$ for every $\nu\in M$;
 \item[(b)] $\nu_0(F)\ge a$ provided $\nu(F)\ge a$ for every $\nu\in M$.
 \end{itemize}
 \end{remark}

The \textit{Maharam type} of a measure $\mu\in P(K)$ can be defined as  the least cardinal number $\kappa$ for which there exists a family $\mathcal{C}\sub Bor(K)$ of  cardinality
$\kappa$ and such that the condition $\inf\{\mu(B\triangle C):\ C\in\mathcal{C}\}=0$ is satisfied for every $B\in Bor(K)$. Equivalently, $\mu$ has Maharam type $\kappa$ if the space of all $\mu$-integrable functions $L_1(\mu)$ has density $\kappa$ as a Banach space.
A measure $\mu\in P(K)$ is \textit{homogeneous} if its type is the same on every set $B\in Bor(K)$ of positive measure.

The following fact is well-known, see e.g.\ Plebanek \cite{Pl95}, Lemma 2, or Fremlin \cite{Fr97}, Introduction.

\begin{lemma}\label{homogeneous}
If a compact space $K$ carries a regular measure of uncountable type, then there is $\mu\in P(K)$ which is homogeneous of type $\omega_1$.
\end{lemma}

Let $\mu\in P(K)$ and denote its measure algebra by $Bor(K)/_{\mu=0}$. For $B\in Bor(K)$ let  $B^\bullet$ stand for the corresponding element of $Bor(K)/_{\mu=0}$.
We shall use the following standard result.

\begin{lemma}\label{indeplemma}
Let $\mu\in P(K)$ be a homogeneous measure of type $\omega_1$ and let $\mathcal{C}$ be a countable family of Borel subsets of $K$.
Then there is $B\in Bor(K)$ such that $\mu(B)=\frac{1}{2}$ and $B$ is $\mu$-independent of $\mathcal{C}$, i.e. $\mu(B\cap C)=\frac{1}{2}\mu(C)$ for every $C\in\mathcal{C}$.
\end{lemma}

\begin{proof}
By the Maharam Theorem (see Maharam \cite{Ma42} or Fremlin \cite{Fr89}) there is a measure-preserving isomorphism of measure algebras
$\varphi:Bor(K)/_{\mu=0}\to \alg$, where $\alg$ is the measure algebra of the usual product measure $\lambda$ on $2^{\omega_1}$.
Let $\mathcal{C}^\bullet=\{C^\bullet:\ C\in\mathcal{C}\}$.

 Recall that for every $a\in\alg$ there is a set $A\sub 2^{\omega_1}$ depending on coordinates in a countable set $I_A\sub \omega_1$
 such that $A^\bullet=a$, see Fremlin \cite{Fr02}, section 8.
 Therefore, there is a countable set $I\sub \omega_1$ such that for every $C\in\CC$ there is $A\sub 2^{\omega_1}$ such that $A=A'\times 2^{\omega_1\sm I}$ for some
 $A'\in Bor(2^I)$ and $\vf(C^\bullet)=A^\bullet$.

 Take $\xi<\omega_1$ such that $\xi>\sup I$, and  $B\in Bor(K)$ for which $B^\bullet=\varphi^{-1}(c_\xi^\bullet)$, where $c_\xi=\{x\in 2^{\omega_1}:\ x(\xi)=0\}$.
 Then $B$ has the required property.
  \end{proof}

The following corollary can be easily obtained using Lemma \ref{indeplemma} and regularity of $\mu$.

\begin{corollary}\label{indepcor} Let $\mu\in P(K)$ be a homogeneous measure of type $\omega_1$. For every $\varepsilon>0$ there exist sequences $\langle B_\xi\in Bor(K): \xi<\omega_1\rangle$ and $\langle U_\xi\in Open(K): \xi<\omega_1\rangle$ such that
\begin{itemize}
	\item[(i)] $\mu(B_\xi)=\frac{1}{2}$,
	\item[(ii)] $B_\xi\subseteq U_\xi$ and $\mu(U_\xi\setminus B_\xi)<\varepsilon$,
	\item[(iii)] $B_\xi$ is $\mu$-independent of the algebra generated by $\mathcal{C}_\xi=\{B_\eta, U_\eta:\ \eta<\xi\}$.
\end{itemize}
\end{corollary}

\section{Measures on $K\times K$}\label{section:3}

In this section we consider a fixed homogeneous measure $\mu\in P(K)$ of type $\omega_1$. Given two algebras $\mathcal{A}$ and $\mathcal{B}$, we write
\[\mathcal{A}\otimes\mathcal{B}=alg(\{A\times B: A\in\mathcal{A},B\in\mathcal{B}\}),\]
for their product algebra;
here $alg(\cdot)$ denotes the algebra of sets generated by a given family. Let $\mathcal{R}$ denote the Borel rectangle algebra in $K\times K$, i.e.
\[\mathcal{R}=Bor(K)\otimes Bor(K).\]
The following notation is crucial for our considerations: given an algebra $\AAA\sub Bor(K)$, we write $P(\AAA\otimes\AAA,\mu)$ for the family of all finitely additive
probability measures $\nu$ on $ \AAA\otimes\AAA$ such that
\[\nu(A\times K)=\nu(K\times A)=\mu(A)\mbox{ for every } A\in\AAA.\]

By a result due to Marczewski and Ryll-Nardzewski \cite{MCRN}  every $\nu\in P(\mathcal{R},\mu)$ is automatically countably additive and can be
extended to a (regular) measure on the product $\sigma$-algebra $\sigma(Bor(K)\otimes Bor(K))$. In turn, such a measure can be extended to a regular measure
on $Bor(K\times K)$.  We outline below a relatively short argument for completeness (cf.\ Plebanek \cite{Pl89}, Theorem 4).

\begin{theorem} \label{rectextlemma}
Every $\nu\in P(\mathcal{R},\mu)$ can be extended to
a regular Borel measure on $K\times K$.
\end{theorem}

\begin{proof}
Let $\LL$ denote the family of finite unions of rectangles $F\times F'$ where $F,F'\subset K$ are closed.
Using the fact that $\nu\in P(\mathcal{R},\mu)$, it is easy to see that $\nu$ is $\LL$-regular, i.e. for every $\varepsilon>0$ and $A\in\mathcal{R}$ there exists $L\in\mathcal{L}$ contained in $A$ and such that $\nu(A\setminus L)<\varepsilon$.

Let $\FF$ be the lattice of all closed subsets of $K\times K$.
By the main result from Bachman and Sultan \cite{BS80}, $\nu$ can be extended to an $\FF$-regular finitely additive measure $\nu'$ on $alg(\mathcal{R}\cup \FF)$.
By $\FF$-regularity and compactness, $\nu'$ is continuous from above at $\emptyset$ and the standard Caratheodory extension procedure gives an extension to a regular measure
on $\sigma(\RR)=Bor(K\times K)$.
\end{proof}

For a subset $B\subseteq K$, we use below the following notation: $B^0=B$ and $B^1=B^c=K\setminus B$. We now prove two lemmas concerning extensions of measures on finite algebras with fixed marginal distributions.

\begin{lemma}\label{extlemma}
If $\mathcal{A}$ is a finite subalgebra of $Bor(K)$ then every 
$\nu\in P(\mathcal{A}\otimes\AAA,\mu)$ can be extended to  $\widehat{\nu}\in P(\mathcal{R},\mu )$.
 \end{lemma}

\begin{proof}
Let us fix a finite algebra $\mathcal{A}\sub Bor(K)$ and  $\nu\in P(\mathcal{A}\otimes\AAA,\mu)$.
Let  $\mathcal{A}_1=alg(\mathcal{A}\cup\{B\})$ where $B\in Bor(K)\setminus\mathcal{A}$.  We shall check first that
$\nu$ can be extended to  $\nu_1\in P(\mathcal{A}_1\otimes\mathcal{A}_1,\mu )$.

Let $\{S_1,\ldots,S_l$\} be the family of all  atoms of $\mathcal{A}$ having positive measure.
It is sufficient to define $\nu_1$ only on atoms of $\mathcal{A}_1\times\mathcal{A}_1$. For $i,j\le l$ let \[\alpha_{i,j}={\nu(S_i\times S_j)}/(\mu(S_i)\mu(S_j)).\]
For $\varepsilon_1,\varepsilon_2\in\{0,1\}$ put
\[\nu_1((S_i\times S_j)\cap(B^{\varepsilon_1}\times B^{\varepsilon_2}))=\mu(S_i\cap B^{\varepsilon_1})\mu(S_j\cap B^{\varepsilon_2})\alpha_{i,j}.\]
It is easy to check that this uniquely defines the required $\nu_1 \in P(\mathcal{A}_1\otimes\mathcal{A}_1,\mu)$ (cf.\ the proof of the next lemma).

It follows that $\nu$ admits an extension to $\nu_\DD\in P(\DD\otimes\DD,\mu)$ for every finite algebra $\DD$ such that $\AAA\sub\DD\sub Bor(K)$. Now the assertion follows by compactness argument since  the set $P(\RR,\mu)$ is closed. in the space $[0,1]^\RR$.
\end{proof}

\begin{lemma}\label{extindeplemma}
 Let $\mathcal{A}$ be a finite subalgebra of $Bor(K)$, $\mathcal{A}_1=alg(\mathcal{A}\cup\{B\})$, where $B\in Bor(K)$ is $\mu$-independent of $\mathcal{A}$ and $\mu(B)={1}/{2}$.

 Then for every $\nu\in P(\mathcal{A}\otimes\mathcal{A},\mu)$ there exists an extension $\nu_1\in P(\mathcal{A}_1\otimes\mathcal{A}_1,\mu)$ of $\nu$ such that $\nu_1(B\times B)=\frac{1}{2}$.
\end{lemma}

\begin{proof}
We extend $\nu$ to $\nu_1\in P(\mathcal{A}_1\times\mathcal{A}_1,\mu)$ in a similar way to the one presented in the proof of Lemma \ref{extlemma}.

Let $T_1,\ldots, T_k$ be the list of all the atoms of $\mathcal{A}$. For all $i,j\le k$ and $\varepsilon_1,\varepsilon_2\in\{0,1\}$ put
\[\textstyle\nu_1((T_i\times T_j)\cap(B^{\varepsilon_1}\times B^{\varepsilon_2}))=\frac{1}{2}\nu(T_i\times T_j)\]
 if $\varepsilon_1=\varepsilon_2$ and $0$ otherwise. Then
\[\textstyle\nu_1(B\times B)=\sum_i\sum_j\nu_1((T_i\times T_j)\cap(B\times B))=\frac{1}{2}\sum_i\sum_j\nu(T_i\times T_j)=\frac{1}{2}.\]
We now prove that $\nu_1\in P(\mathcal{A}_1\otimes\mathcal{A}_1,\mu)$. It is sufficient to check
that $\nu_1(S\times K)=\nu_1(K\times S)=\mu(S)$ for every atom $S$ of the algebra $\mathcal{A}_1$.
We have
\[\textstyle\nu_1((T_i\cap B)\times K)=\sum_j\nu_1((T_i\cap B)\times T_j)=\]
\[=\textstyle\sum_j\left(\nu_1((T_i\cap B)\times(T_j\cap B))+\nu_1((T_i\cap B)\times(T_j\cap B^c))\right)=\]
\[=\textstyle\sum_j\nu_1((T_i\cap B)\times(T_j\cap B))=\frac{1}{2}\sum_j\nu(T_i\times T_j)=\]
\[=\textstyle\frac{1}{2}\nu(T_i\times K)=\frac{1}{2}\mu(T_i)=\mu(T_i\cap B),\]
where the last identity follows from the $\mu$-independence of $B$ and $\mathcal{A}$. Similarly one checks remaining possibilities.
\end{proof}

\begin{lemma}\label{mainlemma} Let $\mu\in P(K)$ be a homogeneous measure of type $\omega_1$ and suppose that
 $\langle B_\xi\in Bor(K): \xi<\omega_1\rangle$,  $\langle U_\xi\in Open(K): \xi<\omega_1\rangle$ and $\mathcal{C}_\xi$ are as in Corollary \ref{indepcor}.

 For every $\xi<\omega_1$ there is $\nu_\xi\in P(\mathcal{R},\mu)$ such that:
\begin{itemize}
	\item $\nu_\xi(B_\eta\times B_\eta)=\frac{1}{2}$ for $\eta\ge\xi$,
	\item $\nu_\xi(A\times A)=(\mu\otimes\mu)(A\times A)$ for every $A\in\mathcal{C}_\xi$.
\end{itemize}
\end{lemma}

\begin{proof}
Fix $\xi<\omega_1$. Let $\mathcal{A}$ be a finite algebra generated by some elements of $\CC_\xi$ and $I$ be a finite subset of $\omega_1\sm\xi$. Then there is
$\nu_{\AAA, I}\in P(\RR,\mu)$ such that

\begin{itemize}
	\item[---] $\nu_\xi(B_\eta\times B_\eta)=\frac{1}{2}$ for $\eta\in I$,
	\item[---] $\nu_\xi(A\times A)=(\mu\otimes\mu)(A\times A)$ for every $A\in\mathcal{A}$.
\end{itemize}

Indeed, such $\nu_{\AAA, I}$ can be first defined on $alg(\AAA\cup\{B_\eta:\eta\in I \})$ using Lemma \ref{extindeplemma} and induction on $|I|$ and then extended
to a member of $P(\RR,\mu)$ using Lemma \ref{extlemma}.

Now the existence of $\nu_\xi$ with the required properties follows again by  compactness argument:  $P(\RR,\mu)$ is clearly a closed
subset of $[0,1]^\RR$, so it is compact in the topology of convergence on all elements of $\RR$. Hence the required measure $\nu_\xi$ can be defined as a  cluster point of
the net $\nu_{\AAA, I}$ indexed by the pairs $(\AAA, I)$, where $\AAA$ is an algebra generated by a finite subset of $\CC_\xi$ and $I$ is a finite subset of $\omega_1\sm\xi$.
\end{proof}

\section{Main result}\label{section:4}

We are now ready to prove our main result.

\begin{theorem}\label{mainthm}
Let $P(K\times K)$ have countable tightness. Then every $\mu\in P(K)$ has countable type.
\end{theorem}

\begin{proof}
Assume for the sake of contradiction that there exists $\mu\in P(K)$ of uncountable type. Without loss of generality we can assume that $\mu$ is a homogeneous measure of
 type $\omega_1$, see Lemma \ref{homogeneous}. Let $0<\varepsilon<1/16$.

 Take the sequences $\langle B_\xi: \xi<\omega_1\rangle$ and $\langle U_\xi: \xi<\omega_1\rangle$ as in Corollary \ref{indepcor}. For every $\xi<\omega_1$ take $\nu_\xi\in P(\mathcal{R},\mu)$ as in Lemma \ref{mainlemma} and extend it to $\widehat{\nu_\xi}\in P(K\times K)$ using  Theorem  \ref{rectextlemma}. Let $\nu\in P(K\times K)$ be a cluster point of the sequence $\langle\widehat{\nu}_\xi: \xi<\omega_1\rangle$, i.e.
 \[\nu\in\bigcap_{\xi<\omega_1}\overline{\{\widehat{\nu}_\eta: \eta\ge\xi\}}.\]
 We shall show that $\nu\notin\overline{\{\widehat{\nu}_\eta: \eta\in I\}}$ for every countable $I\sub \omega_1$, which  will contradict the assumption that the tightness of $P(K\times K)$ is countable.

Let $I\sub\omega_1$ be countable. Take $\xi>\sup I$. By regularity of $\mu$ there exists closed $F\sub K$ such that $F\sub B_\xi$ and $\mu(B_\xi\setminus F)<\varepsilon$.
For every $\eta\in I$ we have $\widehat{\nu}_\eta(B_\xi\times B_\xi)=\frac{1}{2}$, so
\[\widehat{\nu}_\eta((B_\xi\times B_\xi)\setminus(F\times F))\le
\widehat{\nu}_\eta((B_\xi\sm F)\times K)+ \widehat{\nu}_\eta(K\times (B_\xi\sm F))=2\mu(B_\xi\sm F)<2\eps.\]
Therefore $\widehat{\nu}_\eta(F\times F)>\frac{1}{2}-2\varepsilon$ whenever $\eta\in I$.

On the other hand,  if $\eta>\xi$ then 
\[ \nu_\eta(U_\xi\times U_\xi)=(\mu\otimes\mu)(U_\xi\times U_\xi)=\mu(U_\xi)^2<(1/2+\eps)^2,\]
so by Remark \ref{remark}(a)  $\nu(F\times F)\le \nu(U_\xi\times U_\xi)\le (1/2+\eps)^2$.

As $\varepsilon<1/16$, we have  $\left( {1}/{2}+\varepsilon\right)^2< {1}/{2}-2\varepsilon$. We conclude from Remark \ref{remark}(b)  that  $\nu\notin\overline{\{\widehat{\nu}_\eta: \eta\in I\}}$ and
the proof is complete.

\end{proof}

Let us remark that modifying our way to Theorem \ref{mainthm} one can prove the following more general result.

\begin{theorem}\label{mainthm2}
Suppose that $K$ and $L$ are compacta carrying measures of uncountable type. Then $\tau( P(K\times L))\ge\omega_1$.
\end{theorem}

\section{Some consequences and open problems}\label{section:5}

In this final section we present several consequences of Theorem \ref{mainthm} as well as some open problems.

It is not difficult to check that if every $\mu\in P(K)$ has countable type then every $\nu\in P(K\times K)$ has countable type as well.
In connection with Problem \ref{mainproblem} and Theorem \ref{mainthm} it is natural to ask the following.

\begin{problem}\label{ctbltauproblem}
Suppose that $P(K)$ has countable tightness. Does $P(K\times K)$ have  countable tightness?
\end{problem}

As far as we know, the problem is open. Note that $P(K)\times P(K)$ embeds into $P(K\times K)$ and if $\tau(P(K))=\omega$ then $\tau(P(K)\times P(K))=\omega$, since the countable tightness is productive for compact spaces (see Engelking \cite{En}, 3.12.8).
However, the space $P(K\times K)$ seems to be far more complicated  than $P(K)\times P(K)$.

\subsection{Rosenthal compacta}

Recall that a compact space $K$ is said to be Rosenthal compact if $K$ embeds into $B_1(X)$, the space of Baire-one functions on a Polish space $X$ equipped with the topology of pointwise convergence. The class of Rosenthal compacta is stable under taking countable product and, by
a result of Godefroy \cite{Go80},   if $K$ is Rosenthal compact, then so is $P(K)$. Moreover, Rosenthal compacta are Fr\'echet-Urysohn spaces (see Bourgain, Fremlin and Talagrand \cite{BFT}), hence they have countable tightness. This, together with Theorem \ref{mainthm}, implies the result of Bourgain and \stevo\  mentioned in the introductory section.

\begin{corollary}
If $K$ is Rosenthal compact, then every $\mu\in P(K)$ has countable type.
\end{corollary}

\subsection{Property (C) of Corson}

Let $X$ be a Banach space. Corson \cite{Co61} introduced the following convex analogue of the Lindel\"of property:  $X$ is said to have property (C) if for every family $\mathcal{C}$ of convex closed subsets of $X$ we have $\bigcap\mathcal{C}\neq\emptyset$ provided that every countable subfamily of $\mathcal{C}$ has nonempty intersection. For $C(K)$ spaces,
 Pol \cite[ Lemma 3.2]{Po82} gave the following characterization of  property (C).

\begin{theorem}[Pol]\label{polthm}
For a compact space K the following are equivalent:
\begin{enumerate}
	\item the space C(K) has property (C);
	\item for every family $\mathcal{M}\subseteq P(K)$ and every $\mu\in\overline{\mathcal{M}}$ there exists countable subfamily $\mathcal{N}\subseteq\mathcal{M}$ such that $\mu\in\overline{\rm conv}\,\mathcal{N}$.
\end{enumerate}
\end{theorem}

Let us say that $P(K)$ has \textit{convex countable tightness} if $P(K)$ fulfils  condition (2) of Theorem \ref{polthm}. Clearly countable tightness implies convex countable tightness,
Pol \cite{Po82} asked if those properties are actually equivalent, which amounts to asking the following.

\begin{problem}\label{problem_C}
Assume $C(K)$ has property (C). Does this imply the countable tightness of $P(K)$?
\end{problem}

Frankiewicz, Plebanek and Ryll--Nardzewski (\cite{FPR}, Theorem 3.4) answered this question affirmatively assuming Martin's axiom MA$(\omega_1)$. Without any additional set-theoretic assumptions they also obtained the following partial result (\cite{FPR}, Theorem 3.2):

\begin{theorem}\label{fprnthm}
Assume that every $\mu\in P(K)$ has countable type. Then $\tau(P(K))=\omega$ if and only if $C(K)$ has property (C).
\end{theorem}

Looking back at the proof of Theorem \ref{mainthm} it is easy to notice that  we in fact got the following (formally) stronger result.

\begin{theorem}\label{convexthm}
Let $P(K\times K)$ have convex countable tightness. Then every $\mu\in P(K)$ has countable type.
\end{theorem}

Theorem \ref{convexthm} yields the positive solution to the following instance of Problem \ref{problem_C}.

\begin{corollary}
For any compact space $K$, $\tau(P(K\times K))=\omega$ if and only if $C(K\times K)$ has property (C).
\end{corollary}

\begin{proof}
Assume that $\tau(P(K\times K))=\omega$. By Theorem \ref{mainthm}, every $\mu\in P(K)$ has countable type, hence every $\mu\in P(K\times K)$ has countable type. By Theorem \ref{fprnthm}, $C(K\times K)$ has property (C).

For the converse assume that $C(K\times K)$ has property (C). By Theorem \ref{polthm}, $P(K\times K)$ has convex countable tightness, which by Theorem \ref{convexthm} implies that every $\mu\in P(K\times K)$ has countable type. Using Theorem \ref{fprnthm} again we conclude  that $\tau(P(K\times K))=\omega$.
\end{proof}

In connection to Problem \ref{ctbltauproblem}, one can ask the following question on the property (C).

\begin{problem}\label{cproblem}
Let $C(K)$ have property (C). Does $C(K\times K)$ have also property (C)?
\end{problem}

Note that the converse holds true. Indeed, if $X$ is a Banach space with property (C) and $Y$ is its closed subspace, then $Y$ also has property (C). Since $C(K)$ embeds isometrically into $C(K\times K)$ by the operator $C(K)\ni f\mapsto f\circ\pi\in C(K\times K)$, where $\pi:K\times K\to K$ is a projection, $C(K)$ is isometric to a closed subspace of $C(K\times K)$.

\subsection{Topological dichotomy for $P(K\times K)$}

The particular case of Theorem 2.2 of Krupski and Plebanek \cite{KP} states that given a compact space $K$, $P(K)$ contains either a $\mathbb{G}_\delta$ point (i.e. a point of countable character in $P(K)$) or a measure of uncountable type. Thus Theorem \ref{mainthm} immediately implies the following.

\begin{corollary}
For every compact space $K$, either $P(K\times K)$ contains a $\mathbb{G}_\delta$ point or $P(K\times K)$ has uncountable tightness.
\end{corollary}

Recall that  a measure $\mu\in P(K)$ is  countably deteremined (CD)  if there is a countable family $\FF$ of closed subsets of $K$ such that
$\mu(U)=\sup\{\mu(F): F\sub U, F\in\FF\}$ for every open $U\sub K$.
Moreover, $\mu$ is strongly countably determined (SCD) if one can choose such a family $\FF$ consisting of closed $\mathbb{G}_\delta$ sets;
  see \cite{Po82} and \cite{KP} for basic properties of CD and SCD measures and further references. For every $\mu\in P(K)$ we have the following implications
\[\mu \mbox{ is SCD } \Rightarrow \mu \mbox{ is CD}\Rightarrow \mu\mbox{ has countable type}.\]
A measure $\mu\in P(K)$ is strongly countably determined if and only if
$\mu$ is a $\mathbb{G}_\delta$ point of $P(K)$.  Thus the statement  `every $\mu\in P(K)$ is strongly countably determined' is equivalent to saying that $P(K)$ is first-countable.
In the light of our main result, the following problem seems to be natural.

\begin{problem}
Suppose that $P(K)$ or $P(K^\omega)$ is a Fr\'echet-Urysohn space. Is every $\mu\in P(K)$ countably determined?
\end{problem}

It is not known whether every measure on a Rosenthal compactum is countably determined --- see
 Marciszewski and Plebanek  \cite{MP12} for  a partial positive solution.


\begin{thebibliography}{10}
\bibitem{BS80} G.\ Bachman and A.\ Sultan, {\em On regular extensions of measures},
Pacific J.\ Math.\ 86  (1980), 379-610.
\bibitem{Bo76} J.\ Bourgain, {\em Thesis}, Brussels (1974).
\bibitem{BFT} J.\ Bourgain, D.H.\ Fremlin, M.\ Talagrand, {\em Pointwise compact sets
of measurable functions}, Amer.\ J.\ Math.\ 100  (1978), 845--886.
\bibitem{Co61} H.H. Corson, {\em The weak topology of  a  Banach  space},\
Trans.\  Amer.\ Math.\ Soc.\ 101 (1961), 1-15.
\bibitem{En} R.\ Engelking, {\em General Topology}, PWN Warsaw (1977).
\bibitem{FPR} R. Frankiewicz, G.\ Plebanek, C. Ryll--Nardzewski, {\em  Between Lindel\"{o}f property and countable tightness}, Proc.\ Amer.\ Math.\ Soc. 129 (2001), 97--103.
\bibitem{Fr89} D.H.\ Fremlin, {\em Measure algebras}, in: {\em Handbook of
Boolean algebras}, J.D.\ Monk (ed.), North--Holand 1989, Vol. III, Chap.\ 22.
\bibitem{Fr97} D.H.\ Fremlin, {\em On compact spaces carrying Radon
measures of uncountable Maharam type},\ Fund.\ Math.\ 154 (1997), 295--304.
\bibitem{Fr02} D.H.\ Fremlin, {\em Sets determined by few coordinates}, Atti\ Sem.\ Mat.\ Fis.\ Univ.\ Modena 50 (2002), 23--36.
\bibitem{Go80} G.\ Godefroy, {\em Compacts de Rosenthal}, Pacific J.\ Math.\ 91 (1980), 293--306.
\bibitem{KP} M.\ Krupski, G.\ Plebanek, {\em A dichotomy for the convex spaces of probability measures}, Topol. Appl. 158 (2011), 2184--2190.
\bibitem{Ma42} D.\ Maharam, {\em On homogeneous measure algebras}, Proc.\ Nat.\ Acad.\ Sci.\ U.S.A.\ 28 (1942), 108--111.
\bibitem{MP12} W.\ Marciszewski, G.\ Plebanek, {\em On measures on Rosenthal compacta},
J.\ Math.\ Anal.\ Appl.\ 385 (2012), 185-193.
\bibitem{MCRN} E.\ Marczewski, Cz.\ Ryll-Nardzewski, {\em Remarks on the compactness and non-direct products of measures},
Fund.\  Math.\ 40 (1953), 165-170.
\bibitem{Pl89} G.\ Plebanek, {\em Measures on two-dimensional products}, Mathematica\ 36 (1989), 253--258.
\bibitem{Pl95} G.\ Plebanek, {\em On Radon measures on first-countable spaces}, Fund.\ Math.\ 148 (1995), 159--164.
\bibitem{Pl02} G.\ Plebanek,  {\em On compact spaces carrying Radon measures of  large Maharam type},
Acta Univ.\  Car.\  Math.\  Phys.\ 43 (2002), 87--99.
\bibitem{Po82} R.\ Pol, {\em Note on the spaces of regular probability
measures whose topology is determined by countable subsets},\
Pacific J.\ Math.\ 100 (1982), 185--201.
\bibitem{Se71} Z.\ Semadeni, {\em Banach spaces of continuous functions},
PWN Warsaw (1971).
\bibitem{Ta81} M.\ Talagrand, {\em Sur les espace de Banach contenant
$l^{\tau}$}, Israel J.\ Math.\ 40 (1981), 324--330.
\bibitem{To99} S.\ \stevo, {\em Compact sets of the first Baire class},
J.\ Amer.\ Math.\ Soc.\ 12 (1999), 1179--1212.
\end{thebibliography}
\end{document}